\documentclass[a4paper,10pt]{amsart}
\usepackage{amsmath}
\usepackage{amssymb}
\usepackage{amsthm}
\usepackage{enumitem}
\usepackage{latexsym}
\usepackage{mdwlist}

\newtheorem{Def}{Definition}[section]
\newtheorem{Th}[Def]{Theorem}
\newtheorem{Le}[Def]{Lemma}
\newtheorem{Ex}[Def]{Example}

\newtheorem{Co}[Def]{Corollary}

\newtheorem{Pro}[Def]{Proposition}

\begin{document}
\title{A note on conductor ideals}
\author{Andreas Reinhart}
\address{Institut f\"ur Mathematik und wissenschaftliches Rechnen, Karl-Franzens-Universit\"at, NAWI Graz, Heinrichstrasse 36, 8010 Graz, Austria}
\email{andreas.reinhart@uni-graz.at}
\begin{abstract}
Let $S$ be a commutative ring with identity and $R$ a unitary subring of $S$. An ideal $I$ of $S$ is called an $R$-conductor ideal of $S$ if $I=\{x\in S\mid xS\subseteq V\}$ for some intermediate ring $V$ of $R$ and $S$. In this note we present necessary and sufficient criterions for being an $R$-conductor ideal of $S$. We generalize several well known facts about them and present a simple approach to rediscover the results of both old and recent papers. We sketch the boundaries of our criterions by providing a few counterexamples.
\end{abstract}

\maketitle
\leftmargini25pt

\section{Introduction}

Let $K$ be an algebraic number field, $\mathcal{O}_K$ the principal order of $K$ and $\mathcal{O}\subseteq\mathcal{O}_K$ an arbitrary order in $K$. The conductor of the order $\mathcal{O}$ in the principal order $\mathcal{O}_K$ (i.e., $\{x\in\mathcal{O}_K\mid x\mathcal{O}_K\subseteq\mathcal{O}\}$) plays a central role by the investigation of non-principal orders in number fields. For instance, it can be used to describe the behavior of non-unique factorizations in $\mathcal{O}$ (see \cite[Theorems 2.11.9 and 3.7.1]{GHK} for more details). It is obvious that $\{x\in\mathcal{O}_K\mid x\mathcal{O}_K\subseteq\mathcal{O}\}$ is an ideal of $\mathcal{O}_K$. Therefore, it is natural to ask which ideals of $\mathcal{O}_K$ are the conductor of an order in $K$. This question has been answered by P. Furtw\"angler (see \cite{F}), who provided a complete description of these ideals. A few years later, H. Grell (see \cite{Gr}) studied Furtw\"angler's problem in a more general setting. Recently, Furtw\"angler's result attracted some attention \cite{LP,LPP,P} and was proved by using a different method.

In this note we extend the ideas of \cite{F,Gr,LP,P} and prove corresponding results in the setting of commutative rings with identity. Our main result is Theorem 2.7 which easily implies many of the prior results. We replace $\mathcal{O}_K$ by an arbitrary commutative ring $S$ with identity, we replace $\mathbb{Z}$ by a unitary subring $R$ of $S$, and then we ask which ideals of $S$ can appear as the conductor of an intermediate ring of $R$ and $S$ in $S$. In particular, our main result implies \cite[Theorem 1.2]{LP}. We provide necessary and sufficient criteria for being the conductor of an intermediate ring of $R$ and $S$ in $S$. Furthermore, we present a few counterexamples to outline the limitations of our results.

The main result of this work can easily be used to describe the ideals of a commutative Noetherian ring $S$ with identity which are the conductor of a unitary subring of $S$ in $S$. (Note that the prime subring of $S$ is a principal ideal ring.) This characterization could be a starting point for further research on this topic (especially when $S$ is not Noetherian).
\section{Results}

In this work all rings are commutative with an identity element, and every subring of a ring contains the identity of the extension ring. Let $S$ be a ring, $R$ a subring of $S$, and $I$ an ideal of $S$. For $X,Y,Z\subseteq S$ set $(X:_Y Z)=\{x\in Y\mid xZ\subseteq X\}$. We say that $I$ is an $R$-conductor ideal of $S$ if $I=(V:_S S)$ for some intermediate ring $R\subseteq V\subseteq S$. Note that if $I$ is an $R$-conductor ideal of $S$, then $(R:_S S)\subseteq I$. Set $\mathcal{V}(I)=\{M\in {\rm{spec}}(S)\mid I\subseteq M\}$. First we start with a few simple observations.

\begin{Le} Let $S$ be a ring, $R$ a subring of $S$, and $I$ and $J$ ideals of $S$.
\begin{itemize}
\item[\textbf{\textnormal{1.}}] If $R\subseteq V\subseteq S$ is an intermediate ring, then $(R+(V:_S S):_S S)=(V:_S S)$.\\
In particular, $I$ is an $R$-conductor ideal of $S$ if and only if $(R+I:_S S)=I$.
\item[\textbf{\textnormal{2.}}] If $I\subsetneqq J\subseteq (R+I:_S S)$, then $(I:_R (I:_S J))\nsubseteq I$.\\
In particular, $I$ is an $R$-conductor ideal of $S$ if and only if $(I:_R (I:_S (R+I:_S S)))\subseteq I$.
\item[\textbf{\textnormal{3.}}] If $I$ is an $R$-conductor ideal of $S$, then $R+J\subsetneqq S$ or $(I:_R J)\subseteq I$.
\item[\textbf{\textnormal{4.}}] If $I\subsetneqq J$ and $S/I$ is a Noetherian ring, then there is some ideal $L$ of $S$ such that $I\subsetneqq L\subseteq J$, and $(I:_S L)\in\mathcal{V}(I)$.
\end{itemize}
\end{Le}

\begin{proof} \textbf{1.} Let $R\subseteq V\subseteq S$ be an intermediate ring. We need to show that $(R+(V:_S S):_S S)=(V:_S S)$. ``$\subseteq$'': We have $R+(V:_S S)=R+(V:_V S)\subseteq V$, hence $(R+(V:_S S):_S S)\subseteq (V:_S S)$. ``$\supseteq$'': Clearly, $(V:_S S)$ is an ideal of $S$, and thus $(V:_S S)S=(V:_S S)\subseteq R+(V:_S S)$. Consequently, $(V:_S S)\subseteq (R+(V:_S S):_S S)$.\\
\textbf{2.} Let $I\subsetneqq J\subseteq (R+I:_S S)$. It follows that $J=J\cap (R+I)=(J\cap R)+I$, and thus $J\cap R\nsubseteq I$. Since $J(I:_S J)\subseteq I$, we obtain that $J\subseteq (I:_S (I:_S J))$, hence $J\cap R\subseteq (I:_R (I:_S J))$. This implies that $(I:_R (I:_S J))\nsubseteq I$.\\
If $I$ is not an $R$-conductor ideal of $S$, then it follows by 1 that $I\subsetneqq (R+I:_S S)$. Therefore,\linebreak
$(I:_R (I:_S (R+I:_S S)))\nsubseteq I$. Now let $I$ be an $R$-conductor ideal of $S$. By 1 we infer that\linebreak
$(I:_R (I:_S (R+I:_S S)))=(I:_R (I:_S I))=(I:_R S)=I\cap R\subseteq I$.\\
\textbf{3.} Let $I$ be an $R$-conductor ideal of $S$, and $R+J=S$. We have $(I:_R J)S=(I:_R J)(R+J)\subseteq R+I$. Consequently, $(I:_R J)\subseteq (R+I:_S S)=I$.\\
\textbf{4.} Let $I\subsetneqq J$ and $S/I$ a Noetherian ring. There is some ideal $L$ of $S$ such that $I\subsetneqq L\subseteq J$, and $(I:_S L)$ is a maximal element of $\{(I:_S A)\mid A$ is an ideal of $S$ such that $I\subsetneqq A\subseteq J\}$. Observe that $I\subseteq (I:_S L)\subsetneqq S$. We show that $(I:_S L)\in\mathcal{V}(I)$. Let $x,y\in S$ be such that $xy\in (I:_S L)$ and $x\not\in (I:_S L)$. We have $xyL\subseteq I$, and $xL\nsubseteq I$. This implies that $I\subsetneqq I+xL\subseteq L\subseteq J$, hence $(I:_S I+xL)=(I:_S L)$. Since $y(I+xL)=yI+xyL\subseteq I$, we infer that $y\in (I:_S L)$.
\end{proof}

Note that Lemma 2.1.4 can also be proved by using the fact that every module over a Noetherian ring possesses an associated prime ideal.

\begin{Le} Let $S$ be a ring, $R$ a subring of $S$, $T$ a non-empty set, and $(I_i)_{i\in T}$ a family of $R$-conductor ideals of $S$. Then $\bigcap_{i\in T} I_i$ is an $R$-conductor ideal of $S$. In particular, if $n\in\mathbb{N}$, and $(J_i)_{i=1}^n$ is a finite sequence of pairwise coprime $R$-conductor ideals of $S$, then $\prod_{i=1}^n J_i$ is an $R$-conductor ideal of $S$.
\end{Le}

\begin{proof} There is some family $(V_i)_{i\in T}$ of intermediate rings of $R$ and $S$ such that $I_i=(V_i:_S S)$ for all $i\in T$. Observe that $\bigcap_{i\in T} V_i$ is an intermediate ring of $R$ and $S$. Moreover, $\bigcap_{i\in T} I_i=\bigcap_{i\in T} (V_i:_S S)=(\bigcap_{i\in T} V_i:_S S)$, and thus $\bigcap_{i\in I} I_i$ is an $R$-conductor ideal of $S$.
\end{proof}

\begin{Pro} Let $S$ be a ring, $R$ a subring of $S$, and $I$ and $J$ ideals of $S$ such that $(I\cap R)+(J\cap R)=R$.
\begin{itemize}
\item[\textbf{\textnormal{1.}}] $(R+I)\cap (R+J)=R+IJ$, and $(R+I:_S S)(R+J:_S S)\subseteq (R+IJ:_S S)$.
\item[\textbf{\textnormal{2.}}] If $(R+IJ:_S S)=IJ$, then $I(R+J:_S S)=IJ$.\\
In particular, if $IJ$ is an $R$-conductor ideal of $S$, and $I$ is cancellative, then $J$ is an $R$-conductor ideal of $S$.
\end{itemize}
\end{Pro}

\begin{proof} Observe that $I+J=S$, hence $I\cap J=IJ$.\\
\textbf{1.} First we prove that $(R+I)\cap (R+J)=R+IJ$. ``$\subseteq$'': Let $z\in (R+I)\cap (R+J)$. There are some $a,b\in R$, $x\in I$, $y\in J$, $e\in I\cap R$, $f\in J\cap R$ such that $z=a+x=b+y$, and $e+f=1$. We have $z-(af+be)=a+x-(af+be)=(a-b)e+x\in I$, and $z-(af+be)=b+y-(af+be)=(b-a)f+y\in J$. We infer that $z-(af+be)\in I\cap J=IJ$, and thus $z=af+be+z-(af+be)\in R+IJ$. ``$\supseteq$'': Trivial. Next we show that $(R+I:_S S)(R+J:_S S)\subseteq (R+IJ:_S S)$. Let $x\in (R+I:_S S)$, and $y\in (R+J:_S S)$. It follows that $xS\subseteq R+I$, and $yS\subseteq R+J$. This implies that $xyS\subseteq (R+I)\cap (R+J)=R+IJ$, hence $xy\in (R+IJ:_S S)$.\\
\textbf{2.} Let $(R+IJ:_S S)=IJ$. Observe that $I\subseteq (R+I:_S S)$, and $J\subseteq (R+J:_S S)$. Therefore, we have by 1 that $IJ\subseteq I(R+J:_S S)\subseteq (R+I:_S S)(R+J:_S S)\subseteq (R+IJ:_S S)=IJ$, and thus $I(R+J:_S S)=IJ$.
\end{proof}

\begin{Co} Let $S$ be a ring, $R$ a subring of $S$, $n\in\mathbb{N}$, and $(I_i)_{i=1}^n$ a finite sequence of cancellative ideals of $S$ such that $(I_i\cap R)+(I_j\cap R)=R$ for all $i,j\in [1,n]$ such that $i\not=j$. Then $\prod_{i=1}^n I_i$ is an $R$-conductor ideal of $S$ if and only if $I_i$ is an $R$-conductor ideal of $S$ for all $i\in [1,n]$.
\end{Co}

\begin{proof} This is an immediate consequence of Lemma 2.2 and Proposition 2.3.2.
\end{proof}

Observe that if $K$ is an algebraic number field, $\mathcal{O}_K$ is its principal order, and $I$ is a nonzero ideal of $\mathcal{O}_K$, then $\mathbb{Z}+I$ is an order in $K$. In particular, $I$ is a $\mathbb{Z}$-conductor ideal of $\mathcal{O}_K$ if and only if $I$ is the conductor of an order in $K$ by Lemma 2.1.1. Because of this it follows that Corollary 2.4 is a (partial) generalization of \cite[Theorem 1.1]{LP}. The next goal is to characterize maximal ideals of $S$ which are $R$-conductor ideals of $S$.

\begin{Le} Let $S$ be a ring and $R$ a subring of $S$ such that $S$ is a cyclic $R$-module. Then $S=R$.
\end{Le}

\begin{proof} There is some $x\in S$ such that $S=xR$. Therefore, $1=yx$ for some $y\in R$. Moreover, $x^2R=S^2=S=xR$, and thus $S=xR=yx^2R=yxR=R$.
\end{proof}

\begin{Pro} Let $S$ be a ring, $R$ a subring of $S$, $I$ an ideal of $S$, and $M\in\max(S)$.
\begin{itemize}
\item[\textbf{\textnormal{1.}}] $S=R+I$ if and only if $S/I$ is a cyclic $R$-module.
\item[\textbf{\textnormal{2.}}] The following are equivalent:
\begin{itemize}
\item[\textbf{\textnormal{a.}}] $M$ is an $R$-conductor ideal of $S$.
\item[\textbf{\textnormal{b.}}] $R+M\subsetneqq S$.
\item[\textbf{\textnormal{c.}}] $S/M$ is not a cyclic $R$-module.
\item[\textbf{\textnormal{d.}}] If $M\cap R\in\max(R)$, then $\dim_{R/M\cap R}(S/M)\geq 2$.
\end{itemize}
\end{itemize}
\end{Pro}

\begin{proof} \textbf{1.} ``$\Rightarrow$'': Let $S=R+I$. Then $S/I=R+I/I=(1+I)_R$. ``$\Leftarrow$'': Let $S/I$ be a cyclic $R$-module. Then $S/I$ is a cyclic $R+I/I$-module. Lemma 2.5 implies that $S/I=R+I/I$, hence $S=R+I$.\\
\textbf{2.} $\textbf{a.}\Leftrightarrow\textbf{b.}$ By Lemma 2.1.1 we have $M$ is an $R$-conductor ideal of $S$ if and only if $M=(R+M:_S S)$ if and only if $(R+M:_S S)\subsetneqq S$ if and only if $R+M\subsetneqq S$.\\
$\textbf{b.}\Leftrightarrow\textbf{c.}$ This is an immediate consequence of 1.\\
$\textbf{c.}\Rightarrow\textbf{d.}$ Let $M\cap R\in\max(R)$. Since $S/M$ is not a cyclic $R$-module, we have $S/M$ is not a cyclic $R/M\cap R$-vector space, and thus $\dim_{R/M\cap R}(S/M)\geq 2$.\\
$\textbf{d.}\Rightarrow\textbf{b.}$ Case 1: $M\cap R\not\in\max(R)$. Then $R+M/M\cong R/M\cap R$ is not a field, hence $R+M/M\subsetneqq S/M$.\\
Case 2: $M\cap R\in\max(R)$. Then $\dim_{R/M\cap R}(S/M)\geq 2$. Since $R+M/M\cong R/M\cap R$ we infer that $R+M/M\subsetneqq S/M$.\\
In any case we have $R+M\subsetneqq S$.
\end{proof}

Now we are prepared to prove the main result of this paper.

\begin{Th} Let $S$ be a ring, $R$ a subring of $S$, and $I$ an ideal of $S$ such that $S/I$ is a Noetherian ring and $R/I\cap R$ is a principal ideal ring. Then $I$ is an $R$-conductor ideal of $S$ if and only if for every $M\in\mathcal{V}(I)$ we have $R+M\subsetneqq S$ or $(I:_R M)\subseteq I$.
\end{Th}

\begin{proof} First let $I$ be an $R$-conductor ideal of $S$. It follows from Lemma 2.1.3 that for every $M\in\mathcal{V}(I)$ we have $R+M\subsetneqq S$ or $(I:_R M)\subseteq I$.\\
Now let $I$ be not an $R$-conductor ideal of $S$. We show that $R+M=S$ and $(I:_R M)\nsubseteq I$ for some $M\in\mathcal{V}(I)$. Set $L=(R+I:_S S)$. Note that $I\subsetneqq L$ by Lemma 2.1.1. By Lemma 2.1.4 there is some ideal $J$ of $S$ such that $I\subsetneqq J\subseteq L$, and $(I:_S J)\in\mathcal{V}(I)$. Set $M=(I:_S J)$. By Lemma 2.1.2 we obtain that $(I:_R M)\nsubseteq I$. We have $J\cap R/I\cap R=(x+I\cap R)_R$ for some $x\in J\cap R$, and thus $J/I=J\cap (R+I)/I=(J\cap R)+I/I=(x+I)_R$. We infer that $J/I$ is a cyclic $S$-module, hence $S/M\cong_S J/I$. Therefore, $S/M\cong_R J/I$ is a cyclic $R$-module. It follows by Proposition 2.6.1 that $R+M=S$.
\end{proof}

\begin{Co} Let $S$ be a ring, $R$ a subring of $S$, and $I$ an ideal of $S$ such that $S/(R:_S S)$ is a Noetherian ring and $R/(R:_S S)$ is a principal ideal ring. Then $I$ is an $R$-conductor ideal of $S$ if and only if $(R:_S S)\subseteq I$, and for every $M\in\mathcal{V}(I)$ we have $R+M\subsetneqq S$ or $(I:_R M)\subseteq I$.
\end{Co}

\begin{proof} ``$\Rightarrow$'': Obviously, if $I$ is an $R$-conductor ideal of $S$, then $(R:_S S)\subseteq I$. The second statement follows from Lemma 2.1.3.\\
``$\Leftarrow$'': Observe that $(R:_S S)=(R:_R S)\subseteq I\cap R$. Therefore, $S/I$ is an epimorphic image of $S/(R:_S S)$, and $R/I\cap R$ is an epimorphic image of $R/(R:_S S)$. Consequently, $S/I$ is a Noetherian ring and $R/I\cap R$ is a principal ideal ring. The assertion is now an immediate consequence of Theorem 2.7.
\end{proof}

For the sake of completeness we include two sufficient criteria for being an $R$-conductor ideal of $S$. The proof of the following result goes along the same lines as the proof of Theorem 2.7.

\begin{Pro} Let $S$ be a ring, $R$ a subring of $S$, and $I$ an ideal of $S$. Suppose that one of the following conditions holds\textnormal{:}
\begin{itemize}
\item[\textbf{\textnormal{1.}}] $S/I$ is a Noetherian ring, and $(I:_R M)\subseteq I$ for each $M\in\mathcal{V}(I)$.
\item[\textbf{\textnormal{2.}}] $R/I\cap R$ is a principal ideal ring, and $R+M\subsetneqq S$ for each $M\in\mathcal{V}(I)$.
\end{itemize}
Then $I$ is an $R$-conductor ideal of $S$.
\end{Pro}

\begin{proof} Set $L=(R+I:_S S)$. Let $I$ be not an $R$-conductor ideal of $S$. Then $I\subsetneqq L$ by Lemma 2.1.1.\\
First let $S/I$ be a Noetherian ring. By Lemma 2.1.4 there is some ideal $J$ of $S$ such that $I\subsetneqq J\subseteq L$, and $(I:_S J)\in\mathcal{V}(I)$. Set $M=(I:_S J)$. It follows by Lemma 2.1.2 that $(I:_R M)\nsubseteq I$.\\
Next let $R/I\cap R$ be a principal ideal ring. Set $J=(I:_S L)$. There is some $M\in\mathcal{V}(I)$ such that $J\subseteq M$. Note that $L\cap R/I\cap R=(x+I\cap R)_R$ for some $x\in L\cap R$. Therefore, $L/I=L\cap (R+I)/I=(L\cap R)+I/I=(x+I)_R$. We infer that $L/I$ is a cyclic $S$-module, hence $S/J\cong_S L/I$. It follows that $S/J\cong_R L/I$ is a cyclic $R$-module, and thus $S=R+J$ by Proposition 2.6.1. Consequently, $S=R+M$.
\end{proof}

In the following we deal with the question whether Theorem 2.7 is a generalization of \cite[Theorem 1.2]{LP} and Furtw\"angler's result. If $S$ is a Dedekind domain, $I$ is a nonzero ideal of $S$, and $M\in\max(S)$, then let $v_M(I)$ denote the $M$-adic exponent of $I$ (i.e., $v_M(I)$ is the largest $l\in\mathbb{N}_0$ such that $I\subseteq M^l$). The purpose of the next result is to describe the condition ``$(I:_R M)\subseteq I$'' in Theorem 2.7 in more detail.

\begin{Pro} Let $S$ be a Dedekind domain, $R$ a subring of $S$ that is a Dedekind domain, $I$ and $L$ ideals of $S$ such that $I\cap R\not=\{0\}$ and $I\subseteq L$, and $M\in\mathcal{V}(I)$. For $Q\in\mathcal{V}(I)$, set $e_Q=v_Q((Q\cap R)S)$.
\begin{itemize}
\item[\textbf{\textnormal{1.}}] $v_{M\cap R}(M^l\cap R)=\lceil\frac{l}{e_M}\rceil$, for every $l\in\mathbb{N}_0$.
\item[\textbf{\textnormal{2.}}] $v_{M\cap R}(L\cap R)=\max\{\lceil\frac{v_Q(L)}{e_Q}\rceil\mid Q\in\mathcal{V}(I),Q\cap R=M\cap R\}$.
\item[\textbf{\textnormal{3.}}] $(I:_R M)\subseteq I$ if and only if $e_M\nmid (v_M(I)-1)$ or $\frac{v_M(I)-1}{e_M}<\frac{v_Q(I)}{e_Q}$ for some $Q\in\mathcal{V}(I)\setminus\{M\}$ with $Q\cap R=M\cap R$.
\end{itemize}
\end{Pro}

\begin{proof} \textbf{1.} Let $l\in\mathbb{N}_0$. Without restriction let $l\not=0$. Set $m=v_{M\cap R}(M^l\cap R)$. Observe that $M^l\cap R$ is $M\cap R$-primary, hence $M^l\cap R=(M\cap R)^m$. We have $((M\cap R)S)^m=(M^l\cap R)S\subseteq M^l$, hence $l=v_M(M^l)\leq v_M((M\cap R)S)^m)=me_M$. Therefore, $\frac{l}{e_M}\leq m$, and thus $\lceil\frac{l}{e_M}\rceil\leq m$. Since $\lceil\frac{l}{e_M}\rceil e_M\geq l$, we obtain that $(M\cap R)^{\lceil\frac{l}{e_M}\rceil}\subseteq ((M\cap R)S)^{\lceil\frac{l}{e_M}\rceil}\subseteq M^{\lceil\frac{l}{e_M}\rceil e_M}\subseteq M^l$, hence $(M\cap R)^{\lceil\frac{l}{e_M}\rceil}\subseteq M^l\cap R=(M\cap R)^m$. This implies that $m\leq \lceil\frac{l}{e_M}\rceil$. Consequently, $m=\lceil\frac{l}{e_M}\rceil$.\\
\textbf{2.} Note that $L\cap R=(\prod_{Q\in\mathcal{V}(I)} Q^{v_Q(L)})\cap R=(\bigcap_{Q\in\mathcal{V}(I)} Q^{v_Q(L)})\cap R=\bigcap_{Q\in\mathcal{V}(I)} (Q^{v_Q(L)}\cap R)$. We infer by 1 that $v_{M\cap R}(L\cap R)=\max\{v_{M\cap R}(Q^{v_Q(L)}\cap R)\mid Q\in\mathcal{V}(I)\}=\max\{\lceil\frac{v_Q(L)}{e_Q}\rceil\mid Q\in\mathcal{V}(I),Q\cap R=M\cap R\}$.\\
\textbf{3.} Let $P\in\max(R)\setminus\{M\cap R\}$. Clearly, $v_P(M^l\cap R)=0$ for all $l\in\mathbb{N}_0$ (since $M^l\cap R$ is either $R$ or $M\cap R$-primary). We have $J+M=S$, and $I=JM^{v_M(I)}$ for some ideal $J$ of $S$. We infer that $v_P((I:_S M)\cap R)=v_P(JM^{v_M(I)-1}\cap R)=v_P((J\cap R)\cap (M^{v_M(I)-1}\cap R))=\max(v_P(J\cap R),v_P(M^{v_M(I)-1}\cap R))=v_P(J\cap R)=\max(v_P(J\cap R),v_P(M^{v_M(I)}\cap R))=v_P((J\cap R)\cap (M^{v_M(I)}\cap R))=v_P(JM^{v_M(I)}\cap R)=v_P(I\cap R)$.\\
We conclude that $v_P((I:_S M)\cap R)=v_P(I\cap R)$ for all $P\in\max(R)\setminus\{M\cap R\}$. It follows by 2 that
\begin{align*}
&(I:_R M)\subseteq I\Leftrightarrow (I:_S M)\cap R\subseteq I\cap R\Leftrightarrow v_{M\cap R}(I\cap R)\leq v_{M\cap R}((I:_S M)\cap R)\Leftrightarrow\\
&\max\{\lceil\frac{v_Q(I)}{e_Q}\rceil\mid Q\in\mathcal{V}(I),Q\cap R=M\cap R\}\leq\max\{\lceil\frac{v_Q((I:_S M))}{e_Q}\rceil\mid Q\in\mathcal{V}(I),Q\cap R=M\cap R\}\Leftrightarrow\\
&\max\{\lceil\frac{v_Q(I)}{e_Q}\rceil\mid Q\in\mathcal{V}(I),Q\cap R=M\cap R\}\leq\\
&\max(\{\lceil\frac{v_Q(I)}{e_Q}\rceil\mid Q\in\mathcal{V}(I)\setminus\{M\},Q\cap R=M\cap R\}\cup\{\lceil\frac{v_M(I)-1}{e_M}\rceil\})\Leftrightarrow\\
&\lceil\frac{v_M(I)}{e_M}\rceil\leq\max(\{\lceil\frac{v_Q(I)}{e_Q}\rceil\mid Q\in\mathcal{V}(I)\setminus\{M\},Q\cap R=M\cap R\}\cup\{\lceil\frac{v_M(I)-1}{e_M}\rceil\})\Leftrightarrow\\
&\lceil\frac{v_M(I)}{e_M}\rceil\leq\lceil\frac{v_M(I)-1}{e_M}\rceil\textnormal{ or }\lceil\frac{v_M(I)}{e_M}\rceil\leq\lceil\frac{v_Q(I)}{e_Q}\rceil\textnormal{ for some }Q\in\mathcal{V}(I)\setminus\{M\}\textnormal{ with }Q\cap R=M\cap R\Leftrightarrow\\
&e_M\nmid (v_M(I)-1)\textnormal{ or }\lceil\frac{v_M(I)}{e_M}\rceil\leq\lceil\frac{v_Q(I)}{e_Q}\rceil\textnormal{ for some }Q\in\mathcal{V}(I)\setminus\{M\}\textnormal{ with }Q\cap R=M\cap R\Leftrightarrow\\
&e_M\nmid (v_M(I)-1)\textnormal{ or }\frac{v_M(I)-1}{e_M}<\frac{v_Q(I)}{e_Q}\textnormal{ for some }Q\in\mathcal{V}(I)\setminus\{M\}\textnormal{ with }Q\cap R=M\cap R.
\end{align*}
\end{proof}

We provide a few situations where the ``most important ingredient'' of Theorem 2.7 is satisfied.

\begin{Le} Let $S$ be a ring, $R$ a subring of $S$, and $I$ an ideal of $S$ such that $I\cap R\not=\{0\}$. Suppose that one of the following conditions holds\textnormal{:}
\begin{itemize}
\item[\textbf{\textnormal{1.}}] $R$ is a Dedekind domain.
\item[\textbf{\textnormal{2.}}] $R$ is a one-dimensional domain of finite character, and $I$ is a radical ideal of $S$.
\end{itemize}
Then $R/I\cap R$ is a principal ideal ring.
\end{Le}

\begin{proof} Case 1: $R$ is a Dedekind domain. The assertion follows from \cite[Theorem 38.5]{Gi}.\\
Case 2: $R$ is a one-dimensional domain of finite character, and $I$ a radical ideal of $S$. Clearly, $I\cap R$ is a radical ideal of $R$, and thus $I\cap R$ is a finite intersection of maximal ideals of $R$. Therefore, $R/I\cap R$ is isomorphic to a finite direct product of fields, hence $R/I\cap R$ is a principal ideal ring.
\end{proof}

Now we are prepared to rediscover a variant of P. Furtw\"angler's original result.

\begin{Co} Let $S$ be a Dedekind domain, $R$ a subring of $S$ that is a Dedekind domain, and $I$ an ideal of $S$ such that $I\cap R\not=\{0\}$. For $Q\in\mathcal{V}(I)$, set $e_Q=v_Q((Q\cap R)S)$ and $f_Q=\dim_{R/Q\cap R}(S/Q)$. Then $I$ is an $R$-conductor ideal of $S$ if and only if every $M\in\mathcal{V}(I)$ satisfies one of the following conditions\textnormal{:}
\begin{itemize}
\item[\textbf{\textnormal{a.}}] $f_M\geq 2$.
\item[\textbf{\textnormal{b.}}] $e_M\nmid (v_M(I)-1)$.
\item[\textbf{\textnormal{c.}}] $\frac{v_M(I)-1}{e_M}<\frac{v_Q(I)}{e_Q}$ for some $Q\in\mathcal{V}(I)\setminus\{M\}$ with $Q\cap R=M\cap R$.
\end{itemize}
\end{Co}

\begin{proof} This is an easy consequence of Proposition 2.6.2, Theorem 2.7, Proposition 2.10.3, and Lemma 2.11.1.
\end{proof}

\begin{Co} Let $S$ be a Dedekind domain, $R$ a subring of $S$ that is Noetherian and one-dimensional, and $I$ a radical ideal of $S$ such that $I\cap R\not=\{0\}$. For $Q\in\mathcal{V}(I)$, set $f_Q=\dim_{R/Q\cap R}(S/Q)$. Then $I$ is an $R$-conductor ideal of $S$ if and only if for every $M\in\mathcal{V}(I)$ we have $f_M\geq 2$ or $(I:_R M)\subseteq I$.
\end{Co}

\begin{proof} This follows from Proposition 2.6.2, Theorem 2.7 and Lemma 2.11.2.
\end{proof}

The remark after Corollary 2.4 shows that Corollary 2.12 is a generalization of \cite[Theorem 1.2]{LP}. Furthermore, it follows that Theorem 2.7 is an enhancement of \cite[Theorem 1.2]{LP}. Note that Corollary 2.13 is an extension of Corollary 2.12 if we restrict our attention to radical ideals. In the last part of this note we want to point out that most of the conditions in Theorem 2.7, Corollary 2.8, and Proposition 2.9 are crucial, and cannot easily be replaced by weaker properties.

\begin{Ex} There is some ring $S$, some subring $R$ of $S$ such that $R/(R:_S S)$ and $S/(R:_S S)$ are fields, and some ideals $I$ and $J$ of $S$ that are not $R$-conductor ideals of $S$ such that for every $M\in\mathcal{V}(I)$, $R+M\subsetneqq S$, for every $N\in\mathcal{V}(J)$, $R+N\subsetneqq S$ or $(J:_R N)\subseteq J$, $(J:_R A)\subseteq J$ for all ideals $A$ of $S$ such that $J\subsetneqq A$, $S/I$ is an Artinian principal ideal ring, $S/J$ is Noetherian, and $R/I\cap R$ is an Artinian ring whose ideals are 2-generated.
\end{Ex}

\begin{proof} Set $R=\mathbb{Z}[\sqrt{5}]$, $S=\mathbb{Z}[\frac{1+\sqrt{5}}{2}]$, $I=4S$ and $J=\{0\}$. Observe that $S$ is a Dedekind domain, $R$ is a Noetherian one-dimensional domain whose ideals are 2-generated, and $I$ and $J$ are ideals of $S$. Consequently, $S/I$ is an Artinian principal ideal ring, $S/J$ is Noetherian and $R/I\cap R$ is an Artinian ring whose ideals are 2-generated. It is straightforward to prove that $(R:_S S)=2S\in\max(R)\cap\max(S)$. Therefore, $R/(R:_S S)$ and $S/(R:_S S)$ are fields. Furthermore, $(R:_S S)\nsubseteq I$, and $(R:_S S)\nsubseteq J$, hence $I$ and $J$ are not $R$-conductor ideals of $S$. Since $\mathcal{V}(I)=\{(R:_S S)\}$ we infer that for every $M\in\mathcal{V}(I)$, $R+M=R\subsetneqq S$. If $A$ is an ideal of $S$ such that $J\subsetneqq A$, then $(J:_R A)=\{0\}\subseteq J$. Let $N\in\mathcal{V}(J)$. If $N\supsetneqq J$, then $(J:_R N)=\{0\}\subseteq J$, and if $N=J$, then $R+N=R\subsetneqq S$.
\end{proof}

This example shows that the property ``$(R:_S S)\subseteq I$'' in Corollary 2.8 is crucial. Moreover, it shows that ``$R/I\cap R$ is a principal ideal ring'' in Theorem 2.7 and Proposition 2.9.2 cannot be replaced by ``every ideal of $R/I\cap R$ is 2-generated''. The proof of Example 2.14 points out that Corollary 2.13 is not valid without the assumption that $I$ is a radical ideal of $S$. Recall that a ring $R$ is called a B\'ezout ring if every finitely generated ideal of $R$ is principal.

\begin{Ex} There exists some ring $S$, some subring $R$ of $S$, and some ideal $I$ of $S$ that is not an $R$-conductor ideal of $S$ such that $R/I\cap R$ is a zero-dimensional B\'ezout ring, and $(I:_R M)\subseteq I$ for all $M\in\mathcal{V}(I)$.
\end{Ex}

\begin{proof} Let $R$ be a one-dimensional valuation domain that is not Noetherian, $K$ a field of quotients of $R$ and $X$ an indeterminate over $K$. Let $P$ be the maximal ideal of $R$. Clearly, $P$ is not principal, and thus $(R:_K P)=R$. Pick $x\in P\setminus\{0\}$. Set $S=R+XK[X]$, and $I=xR+XK[X]$. Obviously, $S$ is a ring, $R$ is a subring of $S$ and $I$ is an ideal of $S$ such that $I\not=S$. Moreover, $R+I=S$, hence $I\subsetneqq S=(R+I:_S S)$. Therefore, $I$ is not an $R$-conductor ideal of $S$ by Lemma 2.1.1. Observe that $\mathcal{V}(I)=\{P+XK[X]\}$, and $(I:_R P+XK[X])=(xR:_R P)=x(R:_K P)\cap R=xR\subseteq I$. Since $R$ is a valuation domain, it follows that the ideals of $R/I\cap R$ form a chain, hence $R/I\cap R$ is a B\'ezout ring. Since $I\cap R=xR\not=\{0\}$, and $R$ is a one-dimensional domain we obtain that $R/I\cap R$ is zero-dimensional.
\end{proof}

In Lemma 2.1.4, Theorem 2.7, and Proposition 2.9.1 it is possible to replace ``$S/I$ is a Noetherian ring'' by the weaker ``$S/I$ satisfies the ACC on annihilator ideals''. However, Example 2.15 shows that we cannot drop this condition in Proposition 2.9.1.

\bigskip
\textbf{Acknowledgement.} This work was supported by the Austrian Science Fund FWF, Project number P26036-N26. I want to thank the referee for his/her careful reading of the paper.

\end{document}